\documentclass[12pt]{amsart}

\usepackage{scrextend}

\newcommand\legendre[2]{\left(\frac{#1}{#2}\right)}

\makeatletter
\newcommand\footnoteref[1]{\protected@xdef\@thefnmark{\ref{#1}}\@footnotemark}
\makeatother

\usepackage{amsmath,amsfonts,amssymb,amsthm,url,verbatim}
\usepackage[dvips]{graphicx}
\usepackage[croatian,english]{babel}
\usepackage[utf8]{inputenc}
\usepackage[T1]{fontenc}
\usepackage{latexsym}
\usepackage[usenames,dvipsnames]{xcolor}
\usepackage{soul}

\usepackage{booktabs}
\usepackage{longtable}
\usepackage{multirow}

\usepackage[backref]{hyperref}

\usepackage{ifthen}
\usepackage{graphpap}
\usepackage{xparse}
\NewDocumentCommand{\Q}{oo}{%
	\IfValueTF{#2}%
		{\mathbb{Q}_{#1,#2}}%
		{\IfValueTF{#1}%
			{\mathbb{Q}_{\infty,#1}}%
			{\mathbb{Q}}%
		}%
}

\newcommand{\Z}{\mathbb{Z}}
\newcommand{\ZZ}{\mathbb{Z}}

\newcommand{\Qbar}{{\overline{\mathbb Q}}}
\newcommand*{\tors}{\text{tors}}
\DeclareMathOperator{\GL}{GL}
\DeclareMathOperator{\SL}{SL}

\newtheorem{tm}{Theorem}[section]
\newtheorem{proposition}[tm]{Proposition}
\newtheorem{lemma}[tm]{Lemma}
\newtheorem{corollary}[tm]{Corollary}

\theoremstyle{definition}

\theoremstyle{remark}

\newtheorem*{remark}{Remark}

\DeclareMathOperator{\@Gal}{Gal}
\NewDocumentCommand{\Gal}{oo}{%
	\IfValueTF{#2}%
		{\@Gal\left(#1/#2\right)}%
		{\IfValueTF{#1}%
			{}%
			{\@Gal}%
		}%
}

\newcommand{\Aut}{\operatorname{Aut}}

\def\diam#1{\langle#1\rangle}

\title{Torsion groups of elliptic curves over the $\Z_p$-extensions of $\Q$}

\author{Michael Chou}
\address{Dept. of Mathematics, Tufts University, Medford, MA, 02155, USA}
\email{michael.chou@tufts.edu} 
\urladdr{https://sites.tufts.edu/michaelchou/}

\author{Harris B. Daniels}
\address{Department of Mathematics and Statistics, Amherst College, Amherst, MA 01002, USA}
\email{hdaniels@amherst.edu}
\urladdr{https://hdaniels.people.amherst.edu/}

\author{Ivan Krijan}
\address{University of Zagreb, Bijeni\v{c}ka cesta 30, 10000 Zagreb, Croatia}
\email{ikrijan@math.hr}
\urladdr{http://web.math.pmf.unizg.hr/~ikrijan/}
\author{Filip Najman}
\address{University of Zagreb, Bijeni\v{c}ka cesta 30, 10000 Zagreb, Croatia}
\email{fnajman@math.hr}
\urladdr{http://web.math.pmf.unizg.hr/~fnajman/}
\thanks{\thanks{The third and fourth author were supported by the QuantiXLie Center of Excellence, a project
co-financed by the Croatian Government and European Union through the
European Regional Development Fund - the Competitiveness and Cohesion
Operational Programme (Grant KK.01.1.1.01.0004).}}
\date{\today}

\begin{document}
\begin{abstract}We determine, for an elliptic curve $E/\Q$ and for all $p$, all the possible torsion groups \(E(\Q[p])_\tors\), where \(\Q[p]\) is the $\Z_p$-extension of $\Q$.
 \end{abstract}
\maketitle
\section{Introduction}
For a prime number \(p\), denote by \(\Q[p]\) the unique \(\Z_p\)-extension of \(\Q\), and for a positive integer \(n\), denote by \(\Q[n][p]\) the \(n\)\textsuperscript{th} layer of \(\Q[p]\), i.e.\ the unique subfield of \(\Q[p]\) such that $\Gal[\Q[n][p]][\Q]\simeq \Z/p^n\Z$.
	Recall that the \(\Z_p\)-extension of \(\Q\) is the unique Galois extension \(\Q[p]\) of $\Q$ such that
	\[\Gal[\Q[p]][\Q] \simeq \Z_p,\]
	where \(\Z_p\) is the additive group of the \(p\)-adic integers and is constructed as follows. Let
	\[G = \Gal[\Q(\zeta_{p^\infty})][\Q] = \varprojlim\limits_{n}\Gal[\Q(\zeta_{p^{n+1}})][\Q] \stackrel{\sim}{\rightarrow} \varprojlim\limits_{n}(\Z/p^{n+1}\Z)^\times = \Z_p^\times.\]
	Here we know that \(G = \Delta \times \Gamma\), where \(\Gamma \simeq \Z_p\) and \(\Delta \simeq \Z/(p-1)\Z\) for $p\geq 3$ and  \(\Delta \simeq \Z/2\Z\) (generated by complex conjugation) for $p=2$, so we define
	\[\Q[p] := \Q(\zeta_{p^\infty})^\Delta.\]
	We also see that every layer is uniquely determined by
	\[\Q[n][p] = \Q(\zeta_{p^{n+1}})^\Delta,\]
	so for $p\geq 3$ it is the unique subfield of \(\Q(\zeta_{p^{n+1}})\) of degree \(p^n\) over \(\Q\). More details and proofs of these facts about \(\Z_p\)-extensions and Iwasawa theory can be found in \cite[Chapter 13]{washington}.

Iwasawa theory for elliptic curves (see \cite{greenberg2}) studies elliptic curves in $\Z_p$-extensions, in particular the growth of the rank and $n$-Selmer groups in the layers of the $\Z_p$-extensions.

In this paper we completely solve the problem of determining how the torsion of an elliptic curve defined over $\Q$ grows in the $\Z_p$-extensions of $\Q$. These results, interesting in their own right, might also find applications in other problems in Iwasawa theory for elliptic curves and in general. For example, to show that elliptic curves over \(\Q[p]\) are modular for all $p$, Thorne \cite{thorne} needed to show that $E(\Q[p])_\tors=E(\Q)_\tors$ for two particular elliptic curves.

Our results are the following.
	\begin{tm}\label{bigpresult}
	Let \(p \geq 5\) be a prime number, and $E/\Q$ an elliptic curve. Then
	\[E(\Q[p])_\tors = E(\Q)_\tors.\]
	\end{tm}
	\begin{tm}\label{p2result}
	Let $E/\Q$ be an elliptic curve. \(E(\Q[2])_\tors\) is exactly one of the following groups:
	\begin{align*}
	\Z/N\Z, &\qquad 1 \leq N \leq 10 \text{, or } N = 12,\\
	\Z/2\Z \oplus \Z/2N\Z, &\qquad 1 \leq N \leq 4,
	\end{align*}
	and for each group $G$ from the list above there exists an $E/\Q$ such that \(E(\Q[2])_\tors  \simeq G.\)
	
	\end{tm}
	\begin{tm}\label{p3result}
	Let $E/\Q$ be an elliptic curve. \(E(\Q[3])_\tors\) is exactly one of the following groups:
	\begin{align*}
	\Z/N\Z, &\qquad 1 \leq N \leq 10 \text{, or } N = 12, 21 \text{ or } 27,\\
	\Z/2\Z \oplus \Z/2N\Z, &\qquad 1 \leq N \leq 4.
	\end{align*}
		and for each group $G$ from the list above there exists an $E/\Q$ such that \(E(\Q[3])_\tors  \simeq G.\)
	\end{tm}
	\begin{remark}
		By Mazur's theorem \cite{mazur2} we see that
	\begin{align*}
	\{E(\Q[2])_\tors : E/\Q \text{ elliptic curve}\} &= \{E(\Q)_\tors : E/\Q \text{ elliptic curve}\}, \\
	\{E(\Q[3])_\tors : E/\Q \text{ elliptic curve}\} &= \{E(\Q)_\tors : E/\Q \text{ elliptic curve}\} \cup \{\Z/21\Z, \Z/27\Z\}.
	\end{align*}
	However, given a specific $E/\Q$ it is not necessarily the case that $E(\Q_{\infty,p})_\tors = E(\Q)_\tors$. Indeed there are many elliptic curves for which torsion grows from $\Q$ to $\Q_{\infty,p}$, and we investigate this question further in Section \ref{sec:examples}. Specifically, for each prime $p$ we find for which groups $G$ there exists infinitely many $j$-invariants $j$ such that there exists an elliptic curve $E/\Q$ with $j(E)=j$ and such that  $E(\Q)_\tors \subsetneq E(\Q_{\infty,p})_\tors \simeq G$.
	\end{remark}

\section{Notation and auxiliary results}

In this paper we deal with elliptic curves defined over $\Q$, so unless noted otherwise, all elliptic curves will be assumed to be defined over $\Q$.

We will use the following notation throughout the paper: 
\begin{itemize}
	\item For a positive integer \(n\), \(\rho_{E,n}\) is the mod \(n\) Galois representation attached to elliptic curve \(E\); we will write just \(\rho_n\) when it is obvious what \(E\) is.
	\item For a number field $K$, we denote $G_K:=\Gal(\overline K/K)$.
	\item By \(G_{E,K}(n)\) (or just \(G_E(n)\)) we will denote the image (after a choice of basis of $E[n]$) of \(\rho_{E,n}(G_K)\) in $\GL_2(\Z/n\Z)$ i.e.
	\[G_{E,K}(n) = \left\lbrace\rho_{E, n}(\sigma) : \sigma \in \Gal[\overline K][K]\right\rbrace.\]
	\item For a prime number \(\ell \), \(\overline{\rho}_{\ell, E}\) is the \(\ell \)-adic Galois representation and \(T_\ell(E)\) is \(\ell\)-adic Tate module attached to \(E\).
\item We say that an elliptic curve $E$ has or admits an $n$-isogeny over $K$ if there exists an isogeny $f:E\rightarrow E'$ for some elliptic curve $E'$ of degree $n$ with cyclic kernel and such that $E$, $E'$ and $f$ are all defined over $K$, or equivalently if $G_K$ acts on $\ker f$.
\end{itemize}

To make this paper as self-contained as reasonably possible, we now list the most important known results that we will use.

\begin{proposition}\cite[Ch.\ III, Cor.\ 8.1.1]{silverman} \label{prop-weil} Let $E/L$ be an elliptic curve with $L\subseteq\Qbar$. For each integer $n\geq 1$, if $E[n]\subseteq E(L)$ then the $n$\textsuperscript{th} cyclotomic field $\Q(\zeta_n)$ is a subfield of $L$.
\end{proposition}

An immediate consequence of this proposition is

\begin{corollary} \label{cor-weil}
Let \(p\) and \(q\) be odd prime numbers. Then
\[E(\Q[p])[q] \simeq \{O\} \quad \text{or} \quad \Z/q\Z.\]
\end{corollary}

\begin{remark}
We have that \(E[q^n] \nsubseteq E(\Q[p])\), for each positive integer \(n\).
\end{remark}

\begin{proof}
Since \(-1\) and \(1\) are the only roots of unity contained in \(\Q[p]\), by Proposition \ref{prop-weil} \(E[q]\) cannot be contained in \(E(\Q[p])\). The result follows from Proposition \ref{prop-weil}.
\end{proof}

\begin{lemma}\cite[Lemma 4.6]{dlns} \label{lem-j-k_isog}
Let $E$ be an elliptic curve over a number field $K$, let $F$ be a Galois extension of $\Q$, let $p$ be a prime, and let $k$ be the largest integer for which $E[p^k]\subseteq E(F)$.
If $E(F)_\tors$ contains a subgroup isomorphic to $\Z/p^k\Z\oplus\Z/p^j\Z$ with $j\ge k$, then~$E$ admits a $K$-rational $p^{j-k}$-isogeny.
\end{lemma}

\begin{tm}\label{thm-isog}\cite{mazur2,kenku2,kenku3,kenku4,kenku5}
Let $E/\Q$ be an elliptic curve with a rational $n$-isogeny. Then
\[
n\leq 19 \text{ or } n \in\{21,25,27,37,43,67,163\}.
\]
\end{tm}

\begin{corollary}\label{torsiondestroyer}
Let \(p\) be an odd prime number, \(E/\Q\) elliptic curve and \(P \in E(\Q[p])_\tors\) a point of order \(q^n\) for some prime \(q\) and positive integer \(n\), then
\[q^n \in \{2, 3, 4, 5, 7, 8, 9, 11, 13, 16, 17, 19, 25, 27,32, 37, 43, 67, 163\}.\]
\end{corollary}

\begin{proof}
For $q\geq 3$ we have \(E[q] \nsubseteq E(\Q[p])\) by Corollary \ref{cor-weil}, so by Lemma \ref{lem-j-k_isog} we conclude that \(E\) admits a rational \(q^n\)-isogeny.

For $q= 2$, using the same reasoning as before, we conclude that $E$ admits a rational \(2^{n-1}\)-isogeny.

The result now follows from Theorem \ref{thm-isog}.
\end{proof}

\begin{tm}\label{exact_degree}\cite[Theorem 5.7]{gn}
Let $E/\Q$ be an elliptic curve, p a prime and $P$ a point of order $p$ on $E$. Then all of the cases in the table below occur for $p\leq 13$ or $p=37$, and they are the only ones possible.
\[\begin{array}{|c|c|}
\hline
p & [\Q(P):\Q]\\
\hline
2 & 1,2,3\\
\hline
3 & 1,2,3,4,6,8\\
\hline
5 & 1,2,4,5,8,10,16,20,24\\
\hline
7 & 1,2,3,6,7,9,12,14,18,21,24,36,42,48\\
\hline
11 & 5,10,20,{ 40},55,{ 80},100,110,120\\
\hline
13 & 3,4,6,12,{ 24},39,{ 48},52,72,78,96,{ 144},156,168\\

\hline
37 & 12,36,{ 72},444,{ 1296},1332, 1368\\
\hline
\end{array}\]
For all other $p$, for $[\Q(P):\Q]$ the following cases do occur:
\begin{enumerate}
\item \(p^2 - 1\), \hfill for all \(p\),
\item \(8,\ 16,\ 32,\ 136,\ 256,\ 272,\ 288\), \hfill for \(p = 17\),
\item \(\displaystyle \frac{p - 1}{2},\ p-1,\ \frac{p(p-1)}{2},\ p(p-1)\), \hfill if \(p \in \{19,43,67,163\}\),

\item \(2(p-1),\ (p-1)^2\), \hfill if \(p \equiv 1 \pmod{3}\) or \(\legendre{-D}{p} = 1\),\\
							\mbox{} \hfill for some \(D \in \{1,2,7,11,19,43,67,163\}\),
\item \(\displaystyle \frac{(p-1)^2}{3},\ \frac{2(p-1)^2}{3}\), \hfill if \(p \equiv 4, 7 \pmod{9}\),
\item \(\displaystyle \frac{p^2-1}{3},\ \frac{2(p^2 - 1)}{3}\), \hfill if \(p \equiv 2, 5 \pmod{9}\),
\end{enumerate}

Apart from the cases above that have been proven to appear, the only other options that might be possible are:
\[\frac{p^2 - 1}{3},\ \frac{2(p^2 - 1)}{3}, \quad \text{ for } p \equiv 8 \pmod{9}.\]
\end{tm}

\begin{tm}	\cite[Theorem 7.2.]{gn}\label{l-tors}
Let \(p\) be the smallest prime divisor of a positive integer \(d\) and let \(K/\Q\) be a number field of degree \(d\).
\begin{itemize}
\item If \(p \geq 11\), then \(E(K)_{\tors} = E(\Q)_{\tors}\).
\item If \(p = 7\), then \(E(K)[q^\infty] = E(\Q)[q^\infty]\), for all primes \(q \neq 7\).
\item If \(p = 5\), then \(E(K)[q^\infty] = E(\Q)[q^\infty]\), for all primes \(q \neq 5, 7, 11\).
\item If \(p = 3\), then \(E(K)[q^\infty] = E(\Q)[q^\infty]\), for all primes \(q \neq 2,3,5,7,11,13,19,43,67,163\).
\end{itemize}
\end{tm}

We now prove a lemma that we will find useful.

\begin{lemma} \label{ciklicko_prosirenje}
Let $p$ and $q$ be prime numbers such that $q - 1 \nmid p$ and $p \nmid q - 1$. Let $K/\Q$ be a cyclic extension of degree $p$, and $P \in E$ a point of degree $q$. If $P\in E(K)$, then $P \in E(\Q)$.
\end{lemma}

\begin{proof}
If we assume that $\Q(\zeta_q) \subseteq K$, it follows that $q - 1 = [\Q(\zeta_q) : \Q] \mid [K : \Q] = p$, and that is impossible by the assumption that $q - 1 \nmid p$. Therefore, by Corollary \ref{cor-weil} we conclude that $E(K)[q]\simeq \Z/q\Z$.

Let us assume that there is $\sigma \in \Gal[K][\Q]$ such that \(P^\sigma \neq P\) (i.e.\ that \(P \notin E(\Q)\)). That means that there is some \(a \in \{2, 3, \dotsc, q - 1\}\) such that \(P^\sigma = aP\). Furthermore, we know that $\sigma^p = 1$, so
$$P = P^{\sigma^p} = a^pP,$$
which means that $a^p \equiv 1 \pmod{q}$, but there exists such an $a\in \{2, 3, \dotsc, q - 1\}$ if and only if $p \mid q - 1$ or $q - 1 \mid p$, which is a contradiction.
\end{proof}

The following lemma will tell us how far up the tower we have to go to find a point of order $n$, if such a point exists.

\begin{lemma}

Let $E/\Q$ be an elliptic curve and $P\in  E$ a point of order $n$ such that $\Q(P)/\Q$ is Galois and let $E(\Q(P))[n]\simeq \Z/n\Z$. Then $\Gal(\Q(P)/\Q)$ is isomorphic to a subgroup of $(\Z/n\Z)^\times$.
\end{lemma}
\begin{proof}
We see that $G_{\Q}$ acts on $P$ through $G:=\Gal(\Q(P)/\Q)$ so that for any $\sigma \in G_{\Q}$ we have $P^\sigma=aP$, for some $a\in (\Z/n\Z)^\times$, since $|P^\sigma|=|P|.$ Since $G$ acts faithfully on $\diam{P}$, this implies that $G$ is isomorphic to a subgroup of $(\Z/n\Z)^\times$.
\end{proof}

We immediately obtain the following corollary.

\begin{corollary}
\label{kojepolje}
Let $P\in  E$ be a point of order $n$ such that $\Q(P)\subseteq \Q_{\infty,p}$. Then $\Q(P)\subseteq \Q_{m,p}$, where $m=v_p(\phi(n))$.
\end{corollary}

\begin{proposition}
\label{prop_pdiv}
Let $E/F$ be an elliptic curve  over a number field $F$, $n$ a positive integer, $P \in E$ be a point of order $p^{n+1}$ such that $E(F(pP))$ has no points of order $p^{n+1}$ and such that $F(P)/F(pP)$ is Galois. Then $[F(P):F(pP)]$ divides $p^2$.
\end{proposition}
\begin{proof}
Let $Q:=pP$, and consider the equation
 \begin{equation}
pX=Q. \label{eq1}
\end{equation}
So what we're looking for are all the possibilities for $[K(P):K(Q)]$ where $P$ is some solution of \eqref{eq1}.  This degree is the same as the length of the orbit of $P$ under the action of $G:=\Gal(F(P)/F(pP))$ on the solutions of \eqref{eq1}. The $p^2$ solutions of \eqref{eq1} should decompose into orbits with the restriction that if $P$ is defined over a certain field, then so are all multiples of $P$ - so $P$ and all the multiples of P should be in orbits of the same length. If $P$ comes in an orbit of length $n$, and there are $x$ orbits of length $n$ such that all the elements inside are defined over $F(P)$ (and since $F(P)/F(pP)$ is Galois, if one element inside an orbit is defined over $F(P)$, then all of them are), then we have $n\cdot x=p$ or $p^2$ (as the number of solutions \eqref{eq1} is in bijection with $E(F(P))[p]$). Hence $n$ will have to be either $p$ or $p^2$, proving the proposition.
\end{proof}

\begin{remark}
Proposition \ref{prop_pdiv} is a version of \cite[Proposition 4.6.]{gn} with stronger assumptions.
\end{remark}

\section{Proof of Theorem \ref{bigpresult}}

For primes \(p \geq 11\), by Theorem \ref{l-tors} we know that \(E(\Q[n][p])_\tors = E(\Q)_\tors\), for each positive integer \(n\). It follows that \(E(\Q[p])_\tors = E(\Q)_\tors\). It remains to prove this fact for the cases \(p = 7\) and \(p = 5\).

\begin{tm}
$E(\Q[7])_{\tors} = E(\Q)_{\tors}$.
\end{tm}
\begin{proof}
From Theorem \ref{l-tors}, we immediately conclude that $E(\Q[7])[q^\infty]=E(\Q)[q^\infty]$ for all primes $q\neq  7$. It remains to prove that $E(\Q[7])[7^\infty]=E(\Q)[7^\infty]$.

By Corollary \ref{torsiondestroyer} we conclude that there is no $49$-torsion in $E(\Q[7])$, so it remains to prove that $E(\Q[7])[7] = E(\Q)[7]$.

Let $P\in E(\Q[7])$ be a point of order $7$. By Theorem \ref{exact_degree}, $P$ is defined over some field of degree at most $7^2 - 1$. Therefore, $P \in E(\Q[1][7])$. From Lemma \ref{ciklicko_prosirenje} it now follows that $P \in E(\Q)$ and we are done.
\end{proof}

\begin{lemma} \label{nad5nema11}
$E(\Q[5])[11^\infty] = \{O\}$.
\end{lemma}

\begin{proof}
Again by Corollary \ref{torsiondestroyer} we conclude that there is no $121$-torsion in $E(\Q[5])$. It remains to prove that $E(\Q[5])[11] = \{O\}$.

Let $P\in E(\Q[5])$ be a point of order $11$. From Theorem \ref{exact_degree} we conclude that \(P \in E(\Q[1][5])\). The modular curve $X_1(11)$ is the elliptic curve 
\[y^2 + y = x^3 - x^2.\]
We can easily compute (using Magma \cite{magma}) that \(X_1(11)\) has rank \(0\) and torsion \(\Z/5\Z\) over \(\Q[1][5]\), and all the torsion points are cusps, so there are no elliptic curves with $11$-torsion over \(\Q[1][5]\).
\end{proof}

Before proving $E(\Q[5])[5^\infty] = E(\Q)[5^\infty]$ we will need some technical results.

\begin{tm}\cite[Theorem 2]{greenberg}\label{Autrhoprop}
The index \([\Aut_{\Z_5}(T_5(E)) : im(\overline{\rho}_{5, E})]\) isn't divisible by \(25\).
\end{tm}

\begin{lemma}\label{zetaondet}
Let \(n\) be a positive integer and \(\zeta\) an \(n\)\textsuperscript{th} root of unity. Then for every \(\sigma \in \Gal[\Q(E[n])][\Q]\) we have
\[\sigma(\zeta) = \zeta^{\det\rho_n(\sigma)}.\]
\end{lemma}

\begin{proof}
Let \(\{P, Q\}\) be a basis for \(E[n]\) and $e_n(P,Q)=\zeta_n$, where \(\zeta_n\) is a \(n\)\textsuperscript{th} primitive root of unity. For any $n$th root of unity $\zeta$, there exists an \(m \in \Z\) such that \(\zeta = \zeta_n^m\),  So, it suffices to show that \(\sigma(\zeta_n) = \zeta_n^{\det\rho_n(\sigma)}\).

Then there are some \(a, b, c, d \in (\Z/n\Z)^\times\) such that
\[P^\sigma = aP + bQ \quad \text{and} \quad Q^\sigma = cP + dQ.\]
Using properties of the Weil pairing \cite[Ch.\ III, \S 8.]{silverman} we calculate:
\begin{align*}
\sigma(\zeta_n) &= \sigma(e_n(P, Q)) = e_n(P^\sigma, Q^\sigma) = e_n(aP + bQ, cP + dQ)\\
	&= e_n(P,P)^{ac}e_n(P,Q)^{ad}e_n(Q,P)^{bc}e_n(Q,Q)^{bd} \\
	&= 1\cdot \zeta_n^{ad}\cdot \zeta_n^{-bc}\cdot 1^{bd} = \zeta_n^{\det\rho_n(\sigma)}. \qedhere
\end{align*}
\end{proof}

\begin{proposition} \label{prop-detgal}
Let \(n\) be a positive integer and \(\zeta_n\) an \(n\)\textsuperscript{th} primitive root of unity. Let \(K\) be a number field, then
\[\det \rho_{E,n} (G_{K}) \simeq \Gal[\Q(\zeta_n)][K\cap \Q(\zeta_n)].\]
\end{proposition}
\begin{proof}
Let
\[f \colon G_K \to  \Gal[\Q(\zeta_n)][\Q] \qquad \sigma\mapsto g\circ\det\circ \rho_{E,n}(\sigma),\]
where $g$ is the canonical isomorphism mapping \((\Z/n\Z)^\times \to \Gal[\Q(\zeta_n)][\Q]\) by sending \(a\) to \(\sigma_a\), where \(\sigma_a(\zeta_n) = \zeta_n^a\).

As \(\det \rho_{E,n} (G_{K}) \leq (\Z/n\Z)^\times\), it follows that \(f(G_K) \leq \Gal[\Q(\zeta_n)][\Q]\) and hence, by Galois theory,
\(f(G_K) = \Gal[\Q(\zeta_n)][K']\) for some subfield $K'$ of $\Q(\zeta_n)$.
From Proposition \ref{zetaondet} it follows that $f$ is the restriction map sending $\sigma \in G_K$ to \(\left.\sigma\right|_{\Q(\zeta_n)} = \sigma_{\det\rho_n(\sigma)}\).

It follows that $f(G_K)$ comprises exactly those $\sigma_a$ that leave $K\cap \Q(\zeta_n)$ fixed, proving $K'=K\cap \Q(\zeta_n)$ and hence the proposition.
\end{proof}

\begin{lemma} \label{nad5nema5}
$E(\Q[5])[5^\infty] = E(\Q)[5^\infty]$.
\end{lemma}

\begin{proof}
There is no $125$-torsion in $E(\Q[5])$  by Corollary \ref{torsiondestroyer}. If \(P \in E(\Q[5])\) is a point of order \(5\), then \(P\) is defined over \(\Q[1][5]\) by Theorem \ref{exact_degree}, but then by Lemma \ref{ciklicko_prosirenje} it follows that \(P \in E(\Q)\), so it remains to prove that $E(\Q[5])[25] = E(\Q)[25]$.

Let us assume that there is a point \(P \in E(\Q[5])_\tors\) of order \(25\); obviously \(P \notin E(\Q)\). By the previous argumentation, the order 5 point \(5P\) must be defined over \(\Q\). Since \(P \in E(\Q[1][5])\) and the extension \(\Q[1][5]/\Q\) is cyclic, it follows that for every \(\sigma \in G_{\Q}\) there exist some \(a \in (\Z/n\Z)^\times\) such that \(P^\sigma = aP\). Since \(5P \in E(\Q)\), we have \((5P)^\sigma = 5P\), so \(a \equiv 1 \pmod 5\). Hence, \(G_{\Q}(25)\) is of the form
\[\left\lbrace\begin{pmatrix}a & * \\ 0 & *\end{pmatrix} : a \in 1 + 5\Z/25\Z\right\rbrace.\]
Furthermore, \(\det G_{\Q[1][5],E}(25)\) is by Proposition \ref{prop-detgal}, the unique subgroup of \((\Z/25\Z)^\times\) of order \(4\): \(\{7, -1, -7, 1\}\).
The group $G_{\Q[1][5]}$ fixes the point \(P\) so  we conclude that \(G_{\Q[1][5], E}(25)\) is of the form
\[\left\lbrace\begin{pmatrix}1 & * \\ 0 & b\end{pmatrix} : b \in \{7, -1, -7, 1\} \right\rbrace.\]
Since \(G_{\Q[1][5], E}(25)\) is a subgroup of \(G_{\Q}(25)\) of index \(5\), and \(P \notin E(\Q)\), so it follows that
\[G_{\Q}(25) = \left\lbrace\begin{pmatrix}a & * \\ 0 & b\end{pmatrix} : a \in 1 + 5\Z/25\Z,\ b \in \{7, -1, -7, 1\}\right\rbrace.\]
Finally, we calculate that \(600 = [\GL_2(\Z/25\Z) : G_E(25)] \mid [\Aut_{\Z_5}(T_5(E)) : im(\overline{\rho}_{5, E})]\), a contradiction with Theorem \ref{Autrhoprop}.
\end{proof}

\begin{tm}
$E(\Q[5])_{\tors} = E(\Q)_{\tors}$.
\end{tm}
\begin{proof}
From Theorem \ref{l-tors}, we immediately conclude that $E(\Q[5])[q^\infty]=E(\Q)[q^\infty]$ for all primes $q\neq  5, 7, 11$.

There is no \(49\)-torsion in $E(\Q[5])$  by Corollary \ref{torsiondestroyer}, so it remains to prove that $E(\Q[5])[7] = E(\Q)[7]$. Let $P \in E(\Q[5])$ be a point of degree $7$ such that $P \notin E(\Q)$. By Theorem \ref{exact_degree} we conclude that $\gcd(5, [\Q(P) : \Q]) = 1$, which is a contradiction.

The cases \(q = 11\) and \(q = 5\) follow from Lemmas \ref{nad5nema11} and \ref{nad5nema5}.
\end{proof}

\section{Proof of Theorem \ref{p2result}}
Combining Theorem \ref{exact_degree}, Corollary \ref{torsiondestroyer} and the following easy observations:
\begin{itemize}
\item \(p^2 - 1 = (p - 1)(p + 1)\) isn't a power of \(2\) for primes \(p > 3\). This follows from the fact that \(\gcd(p-1, p+1) = 2\),
\item \(p-1\) isn't a power of \(2\) for \(p \in \{19, 43, 67, 163\}\),
\item for a prime $p$, \(p - 1\) is a power of $2$ if and only if \(p\) is of the form \(2^{2^k} + 1\),
\item \(3 \mid \dfrac{(p-1)^2}{3}\) for \(p \equiv 4, 7 \pmod{9}\),
\item \(\dfrac{p^2-1}{3}\) isn't a power of \(2\) for \(p \equiv 2, 5 \pmod{9}\). This follows from the fact that \(p\) is of the form \(3k - 1\) for some integer \(k \geq 6\), and that means that \(\dfrac{p^2 - 1}{3} = (3k-2)\cdot k\), but \(\gcd(3k-2, k) \leq 2\),
\item \(3 \mid \dfrac{p^2 - 1}{3}\) for \(p \equiv 8 \pmod{9}\),
\end{itemize}
we can conclude that \(E(\Q[2])[q]=E(\Q)[q]\) for all primes \(q \neq 2, 3, 5, 7, 13, 17\).

\begin{lemma}
$E(\Q[2])$ does not contain a point of order $13$.
\end{lemma}
\begin{proof}
Suppose that there exists such a curve, then the point $P$ of order $13$ is defined over the quartic field $\Q[2][2] = \Q\left(\sqrt{2 + \sqrt{2}}\right)$ by Theorem \ref{exact_degree}. But then either $E$ or a twist of $E$ would have 13-torsion over $\Q[1][2] = \Q\left(\sqrt 2\right)$, which is not possible by \cite[Theorem 3]{najman9}.
\end{proof}

\begin{lemma}
$E(\Q[2])$ does not contain a point of order $17$.
\end{lemma}
\begin{proof}
By Corollary \ref{kojepolje} and Theorem \ref{exact_degree}, a point $P$ of order $17$ can be defined over a number field of degree $8$ or $16$. An elliptic curve $E/\Q$ with such a point has a $17$-isogeny over $\Q$, so $j(E)=-(17^2\cdot 101^3)/2$ or $j(E)=-(17\cdot 373^3)/2^{17}$. We factor the $17$th division polynomials of an elliptic curve with each of these invariants (the choice of the exact quadratic twist we choose with each $j$-invariant will be irrelevant) over $\Q[3][2]$, and obtain that in one case the smallest degree of an irreducible factor is 4, while in the other case the smallest degree of an irreducible factor is 8.
\end{proof}

So we have \(E(\Q[2])[q^\infty]=E(\Q)[q^\infty]\) for all primes \(q \neq 2, 3, 5, 7\).

\begin{lemma}
If $E(\Q[2])_\tors$ has a point of order $7$, then \(E(\Q[2])_\tors \simeq \Z/7\Z\)
\end{lemma}
\begin{proof}
If $E(\Q[2])_\tors$ has a point of order $7$, then $E/\Q$ has by Corollary \ref{kojepolje} the torsion point $P$ of order 7 is defined over $\Q[1][2]$. Hence either $E(\Q)$ or $E^{(2)}(\Q)$ has a point of order $7$.

If $E(\Q[2])_\tors$ contained a point of order $2$, then both $E(\Q)$ and $E^{(2)}(\Q)$ would also have to contain a point of order $2$, which would mean that one of these groups has a point of order 14, which is by Mazur's theorem impossible.

If $E(\Q[2])_\tors$ contained a point of order $3$, then by Corollary \ref{kojepolje} it would have to be defined over $\Q[1][2]$, which would mean that $E(\Q[1][2])$ has a point of order 21, which is impossible by \cite{kam,km}.

If $E(\Q[2])_\tors$ contained a point of order $5$, then $E/\Q$ would a 35-isogeny, contradicting Theorem \ref{thm-isog}.
\end{proof}

\begin{lemma}
If $E(\Q[2])_\tors$ has a point of order $5$, then \(E(\Q[2])_\tors \simeq \Z/5\Z\) or \(
\Z/10\Z\)
\end{lemma}
\begin{proof}
Suppose $E(\Q[2])_\tors$ has a point of order $5$. By Corollary \ref{kojepolje}, the point of order $P$ on $E$ has to be defined over $\Q[2][2]$.

If $E(\Q[2])_\tors$ had a point of order 15, then both a point of order 3 and of order 5 have to be defined over $\Q[2][2]$, so the 15-torsion point must already be defined over $\Q[2][2]$. But $X_1(15)(\Q)= X_1(15)(\Q[2][2])$, so there are no elliptic curves with a point of order 15 over $\Q[2][2]$.

Suppose now $E(\Q[2])_\tors$ has a point of order 10. Then $E(\Q)$ has a point of order 2. Suppose $E(\Q[2])_\tors\supseteq \Z/2\Z \oplus \Z/10\Z$; then it would follow that $E[2]$ is defined over $\Q[1][2]$ (as it has to be defined over some quadratic field if there is a point of order 2 over $\Q$). Since $E(\Q[2][2])[5]\simeq \Z/5\Z$ and $E(\Q[2][2])[5]\simeq E(\Q[1][2])[5] \oplus E^{(\delta)}(\Q[1][2])[5]$, for some quadratic twist (over $\Q[1][2]$) $E^{(\delta)}$ of $E$. So, since quadratic twisting does not change the $2$-torsion, either $E^{(\delta)}(\Q[1][2])\supseteq \Z/2\Z \oplus \Z/10\Z$ or $E(\Q[1][2])\supseteq \Z/2\Z \oplus \Z/10\Z$. But we compute $X_1(2,10)(\Q(\sqrt 2))=X_1(2,10)(\Q)$, so there are no elliptic curves over $\Q[1][2]=\Q(\sqrt 2)$ with $\Z/2\Z \oplus \Z/10\Z$ torsion.

Finally $E(\Q[2])[20]$ cannot be isomorphic to $\Z/20\Z$ as then $E$ would have a rational $20$-isogeny, which is impossible by Theorem \ref{thm-isog}.
\end{proof}

\begin{lemma}
If $E(\Q[2])_\tors$ has a point of order $9$, then \(E(\Q[2])_\tors \simeq \Z/9\Z\).
\end{lemma}
\begin{proof}
Suppose \(E(\Q[2])_\tors \) contained a point $P$ of order 27. The point $3P$ of order $9$ would have to be defined over $\Q[1][2]$ by Corollary \ref{kojepolje}. There cannot be any 27-torsion over $\Q[1][2]$ \cite{kam,km}. But on the other hand, the point of order $9$ cannot become divisible in a Galois extension of $\Q[1][2]$ of degree $2^n$ by Proposition \ref{prop_pdiv}.

Suppose \(E(\Q[2])_\tors \) contained a point of order 18. Then \(E(\Q[2])_\tors \) would have to contain a 2-torsion point, and a point $P$ of order $9$ would have to be defined over $\Q[1][2]$ by Corollary \ref{kojepolje}. So $E(\Q[1][2])$ would contain a point of order $18$. But there are no elliptic curves defined over $\Q$ with a point of order $18$ over a quadratic field by \cite[Theorem 2]{najman}.
\end{proof}

\begin{lemma}
If $E(\Q[2])_\tors$ has a point of order $12$, then \(E(\Q[2])_\tors \simeq \Z/12\Z\).
\end{lemma}
\begin{proof}
By the previous Lemma, \(E(\Q[2])_\tors\) cannot contain a point of order 9, so it remains to show that $E(\Q[2])_\tors$ does not have full $2$-torsion and that it doesn't have a point of order 24. The fact that there are no points of order 24 follows from the fact that there are no points of order 24 over $\Q^{ab}$ of which $\Q[2]$ is a subfield by \cite[Theorem 1.2]{chou2}.

Suppose now that $E(\Q[2])_\tors\simeq \Z/2\Z \oplus \Z/12\Z$, and let $E(\Q[2])_\tors=\diam{P,Q}$, where $P$ is of order $12$ and $Q$ is of order $2$. We have $2E(\Q[2])_\tors=\diam{2P}\simeq \Z/6\Z$ is a $G_{\Q}$-invariant subgroup, and hence $(6P)^\sigma=6P$ for all $\sigma \in G_{\Q}$, i.e. $6P$ is defined over $\Q$. By \cite[Proposition 4.7]{gn}, $3P$ has to be defined over an extension of $\Q$ which has Galois group $\Z/2\Z, \Z/2\Z \oplus \Z/2\Z$ or $D_4$. So we conclude that it has to be $\Z/2\Z$ (since by assumption $\Q(3P)$ is defined over $\Q[2]$). So the point $3P$ of order 4 is defined over $\Q[1][2]$.

The point $4P$ is also defined over $\Q[1][2]$ by Corollary \ref{kojepolje}. Since $6P$ is defined over $\Q$, then $Q$ has to be defined over $\Q[1][2]$.

So since $4P, 3P$ and $Q$ are all defined over $\Q[1][2]$, we conclude that $E(\Q[1][2])_\tors\simeq \Z/2\Z  \oplus \Z/12\Z$, but this is impossible by \cite[Theorem 10.]{najman9}.
\end{proof}

Finally, we have the following result that controls the $2$-power torsion.

\begin{tm}\cite[Theorem 1]{fuj}
For an elliptic curve $E/\Q$, $E(\Q[2])[2^\infty]\subseteq \Z/2\Z \oplus \Z/8\Z$.
\end{tm}

These results combined prove Theorem \ref{p2result}.

\section{Proof of Theorem \ref{p3result}}\label{proofofp3result}
Combining Theorem \ref{exact_degree}, Corollary \ref{torsiondestroyer} and the following facts:
\begin{itemize}
\item \(p^2 - 1\) isn't a power of \(3\) for any odd prime \(p\) (it's divisible by \(2\)),
\item \(p-1\) isn't a power of \(3\) for any odd prime \(p\) (it's divisible by \(2\)),
\end{itemize}
we can conclude that \(E(\Q[3])[q]=E(\Q)[q]\) for all primes \(q \neq 2, 3, 7, 13, 19\).

\begin{lemma}
\(E(\Q[3])[19]=\{O\}.\)
\end{lemma}
\begin{proof}
From Corollary \ref{kojepolje} we deduce that a point of order \(19\) on $E/\Q$ must be defined over \(\Q[2][3]\). Also $E$ must have a rational $19$-isogeny. There is only one family of quadratic twists (with $j$-invariant $-2^{15}\cdot 3^3$), with complex multiplication by $\Z[(1+\sqrt{-19})/2]$. We check that the \(19\)\textsuperscript{th} division polynomials of these elliptic curves with \(19\)-isogeny don't have a root over the field \(\Q[2][3]\). It is enough to check this for one curve with $j$-invariant $-2^{15}\cdot 3^3$, as if  the \(19\)\textsuperscript{th} division polynomial of this one curve with this $j$-invariant doesn't have a root over \(\Q[2][3]\), then neither does any quadratic twist of $E$. So \(E(\Q[3])[19]=\{O\}\).
\end{proof}

\begin{lemma}
\(E(\Q[3])[13]=\{O\}\)
\end{lemma}
\begin{proof}
From Corollary \ref{kojepolje}, a point of order \(13\) can be defined only over \(\Q[1][3] = \Q(\zeta_9)^+\) (the maximal real subfield of \(\Q(\zeta_9)\)). The modular curve \(X_1(13)\) is a curve of genus \(2\) with the following model (as we can see in \cite{najman9}):
\[y^2 = x^6 - 2x^5 + x^4 - 2x^3 + 6x^2 - 4x + 1.\]
The rank of the Jacobian of this curve over \(\Q(\zeta_9)^+\) is $0$, and the torsion is \(\Z/19\Z\), and we easily check that $X_1(13)(\Q[1][3])=X_1(13)(\Q)$, so there are no elliptic curves with $13$-torsion over $\Q[1][3]$. So \(E(\Q[3])[13]=\{O\}\).
\end{proof}

\begin{lemma}
$E(\Q[3])[5^\infty]=E(\Q)[5^\infty].$
\end{lemma}
\begin{proof}
From Theorem \ref{exact_degree}, we see that if $E(\Q)[5]=\{O\}$, then $E(\Q[3])[5]=\{O\}$. If $E(\Q)[5]\neq \{O\}$, then by Proposition \ref{prop_pdiv}, we see that $E(\Q)[5^\infty]=E(\Q[3])[5^\infty].$
\end{proof}

\begin{lemma}
If $E(\Q[3])_\tors$ has a point of order $7$, then $E(\Q[3])_\tors\simeq \Z/7\Z$ or $\Z/21\Z$
\end{lemma}
\begin{proof}
By Proposition \ref{prop_pdiv}, the $7$-power torsion can grow only if $E(\Q)[7]=\{O\}$. We now determine, for an $E/\Q$ such that $E(\Q)[7]=\{O\}$ and $E(\Q[3])[7]\neq\{O\}$, what are the possible torsion groups of $E(\Q[3])$. By Corollary \ref{kojepolje} we conclude that a point of order 7 appears over $\Q[1][3]$.

We first note that $E(\Q[3])[7^\infty]\simeq \Z/7\Z$, as there cannot be any $49$-torsion by Proposition \ref{torsiondestroyer}. Also $E(\Q[3])$ obviously cannot contain a subgroup isomorphic to $\Z/35\Z$ due to Theorem \ref{thm-isog}.

Suppose $E(\Q[3])$ contains $\Z/14\Z$. Then $E(\Q)$ has a point of order 2, and so $E(\Q[1][3])$ contains a subgroup isomorphic to $\Z/14\Z$. But we compute that $X_1(14)(\Q[1][3])=X_1(14)(\Q)$, which shows that this is impossible.

Suppose $E(\Q[3])$ contains $\Z/21\Z$. First note that then $E(\Q[3])_\tors\simeq \Z/21\Z$, as a larger torsion group would contradict Theorem \ref{thm-isog}. By Corollary \ref{kojepolje}, we see that the $21$-torsion point has to be defined over $\Q[1][3]$.  By \cite[Theorem 1]{najman}, there is a unique such curve $E=162b1$ satisfying this property.
\end{proof}

\begin{lemma}
If $E(\Q)[2^\infty]\neq E(\Q[3])[2^\infty]$, then $E(\Q[3])[2^\infty]\simeq \Z/2\Z \oplus \Z/2\Z$.
\end{lemma}
\begin{proof}
By \cite[Lemma 1]{najman2}, if $E(\Q)[2]\neq \{O\}$, then $E(\Q)[2^\infty]=E(\Q[3])[2^\infty]$ and the $2$-power torsion can grow only $E(\Q)[2]=0$ and all the growth that occurs, occurs over the field obtained by adjoining a $2$-torsion point. By \cite[Proposition 9]{najman}, if $E(\Q[1][3])[2]\neq 0$ then $E(\Q[1][3])[2^\infty]\simeq \Z/2\Z \oplus \Z/2\Z$ and so $E(\Q[3])[2^\infty]\simeq \Z/2\Z \oplus \Z/2\Z$.
\end{proof}

\begin{lemma}
There are no points of order 18 in $E(\Q[3])$.
\end{lemma}
\begin{proof}
Suppose that $E/\Q$ has a point of order 18 over $\Q[3]$. By Corollary \ref{kojepolje}, we see that the points of order 18 are defined over $\Q[1][3]=\Q(\zeta_9)^+$. We will prove that $X_1(18)(\Q(\zeta_9)^+)$ consists of only cusps. We compute that the rank of $J_1(18)(\Q(\zeta_9)^+)$ is $0$. By considering reduction modulo small primes of good reduction, we obtain that the torsion of $J_1(18)(\Q(\zeta_9)^+)$ is a subgroup of $\Z/21\Z^2$. We find 12 points in $X_1(18)(\Q(\zeta_9)^+)$, all of which are cusps, and the differences of pairs of these cusps generate a group isomorphic to $\Z/7\Z \oplus \Z/21\Z$. To prove that $J_1(18)(\Q(\zeta_9)^+)\simeq \Z/7\Z \oplus \Z/21\Z,$ we use the following argument. One finds that $$J_1(18)(\Q)_{\tors}\simeq \Z/21\Z,$$
$$J_1(18)(\Q(\zeta_3))_{\tors}\simeq \Z/3\Z \oplus \Z/21\Z,$$
and by considering reduction modulo small primes, that
$$J_1(18)(\Q(\zeta_3))\leq  \Z/21\Z \oplus \Z/21\Z.$$
Since $\Q(\zeta_9) = \Q(\zeta_3)\Q(\zeta_9)^+$ and $\Q(\zeta_3)\cap \Q(\zeta_9)^+=\Q$, we see that the field of definition of the non-rational elements of $J_1(18)(\Q(\zeta_9))[3]$ is $\Q(\zeta_3)$. Hence $J_1(18)(\Q(\zeta_9)^+)[3]\simeq \Z/3\Z$ and after checking that none of the points in  $J_1(18)(\Q(\zeta_9)^+)$ come from points in $X_1(18)(\Q(\zeta_9)^+)$ apart from the 12 known ones, we are done.

\end{proof}

\begin{proof}[Proof of Theorem \ref{p3result}]
It remains to determine when the $3$-power torsion grows. 

By Lemma \ref{ciklicko_prosirenje}, if $E(\Q[3])[3]\neq 0$, then $E(\Q)[3]\neq 0$.

Suppose $E(\Q[3])$ has a point $P$ of order 27; then $E$ has a rational $27$-isogeny over $\Q$, so $j(E)=-2^{15} \cdot 3 \cdot 53$. By Corollary \ref{kojepolje}, $P\in E(\Q[2][3])$. Let $E=27a2$; we have $j(E)=-2^{15} \cdot 3 \cdot 53$  We factor $\psi:=\psi_{27}/\psi_{9}$, where $\psi_n$ denotes the $n$-division polynomial $E$ - the polynomial which is the product of all the $x$-coordinates of the points of order 27 of $E$ and obtain that this polynomial has roots over $\Q[2][3]$. This implies that there exists a single quadratic twist (over $\Q[2][3]$) $E^\delta$ of $E$, for some $\delta \in L^*/(L^*)^2$ such that $E^\delta(\Q[2][3])$ has a point of order 27. It remains to check whether $E^\delta$ is defined over $\Q$, or equivalently, whether $\delta\cdot u^2=d$ for some $u\in L^*$ and some $d\in \Q^*/(\Q^*)^2$. We compute a $\delta$ and obtain that $N_{\Q[2][3]/\Q}(\delta)=3^{49}$, so the only twists that we can consider are $d=3$ and $-3$. We obtain that $E^{-3}$, which is 27a4 has a point of order 27 over $\Q[2][3]$. Note that from our argumentation it follows that 27a4 is the only elliptic curve with a point of order 27 over $\Q[3]$. For $E=27a4$, we have that $E(\Q[3])_\tors\simeq \Z/27\Z$, as any larger torsion would violate Theorem \ref{thm-isog}.

If $E(\Q)[3^\infty]\simeq \Z/9\Z$, then  we claim that $E(\Q[3])_\tors\simeq \Z/9\Z$. Indeed there cannot be any $q$-torsion for any $q\neq 2,3$ as this would force the existence of a $9q$-isogeny over $\Q$. It is impossible that $E(\Q[3])$ gains any $2$-torsion, as this would imply that there is $\Z/2\Z \oplus \Z/18\Z$ torsion over $\Q[1][3]$, which is impossible by \cite[Theorem 1]{najman}.

Finally from what we have already proved, when $E(\Q[3])[3^\infty]\simeq \Z/3\Z$ then $E(\Q[3])_\tors$ has to be either $\Z/3\Z$, $\Z/6\Z$, $\Z/12\Z$, $\Z/21\Z$ or $\Z/2\Z \oplus \Z/6\Z$.

This, together with the previous lemmas, completes the proof of the Theorem.
\end{proof}



\section{Examples of torsion growth}\label{sec:examples}

In this last section, we address the following question. Fix a prime $p$. Given a group $G$ that can appear as $E(\Q_{\infty,p})_\tors$ for some $E/\Q$, do there exist infinitely many $j$-invariants such that there exists an $E/\Q$ with such a $j$-invariant with $E(\Q_{\infty,p})_\tors \simeq G$ but $E(\Q)_\tors \not\simeq G$? By Theorem \ref{bigpresult} we need only consider the cases where $p = 2,3$.

\begin{tm}\label{thm:p2infinite}
Let $G$ be one of the following groups:
	\begin{align*}
	\Z/N\Z, &\qquad 3 \leq N \leq 10 \text{, or } N = 12,\\
	\Z/2\Z \oplus \Z/2N\Z, &\qquad 1 \leq N \leq 4,
	\end{align*}
There exist infinitely many elliptic curves $E/\Q$ with distinct $j$-invariants such that $E(\Q_{\infty,2})_\tors \simeq G$ and $E(\Q)_\tors \not\simeq G$.
\end{tm}
\begin{proof}
We break this down into cases depending on whether $E(\Q)[2]$ needs to be trivial or not. 

Suppose that $G = \Z/N\Z$ for some odd integer $N$. Then there exist infinitely many elliptic curves $E/\Q$ with distinct $j$-invariants such that $E(\Q)_\tors \simeq \Z/N\Z$ and $E$ has no additional isogenies over $\Q$. This is true because for each $N$ in the statement the elliptic curves $E/\Q$ with $E(\Q)_\tors \simeq \Z/N\Z$ come in a non-isotrivial 1-parameter family that generically doesn't have any additional isogenies. So by Hilbert irreducibility, outside of a ``thin'' set every curve in the family also has no additional isogenies. For more details about Hilbert irreducibility and thin sets, see \cite[Chapter 9]{serre}. Thus, for each of these $E$ the quadratic twist $E^2$ of $E$ by 2 will have trivial torsion over $\Q$, because for odd $N$ we have that $E(\Q(\sqrt{2}))[N] \simeq E(\Q)[N] \oplus E^2(\Q)[N]$ and since $\Q(\sqrt{2})$ does not contain any $m^{th}$-roots of unity for any $2 < m \mid N$ the existence of the Weil-pairing gives that $E^2(\Q)[N] = \{ \mathcal{O} \}$. Further, since $E$ and $E^2$ become isomorphic over $\Q(\sqrt{2}) \subseteq \Q_{\infty,2}$, it follows that $E^2(\Q_{\infty,2})_\tors\simeq \Z/N\Z$. Notice that the torsion can't grow any further since $E$ and hence $E^2$ don't have any additional isogenies. 

Next, suppose that $G = \Z/2n\Z$ with $n\geq 2$. Again there are infinitely many elliptic curves $E/\Q$ such that $E(\Q)_\tors \simeq \Z/2n\Z$ and $E$ has no additional isogenies over $\Q$. Now we have $E^2(\Q)_\tors\simeq \Z/2\Z$ and $E^2(\Q_{2,\infty})_\tors\simeq \Z/2n\Z$ or $\Z/2\Z\oplus\Z/2n\Z$. Again, the torsion can't grow any further since $E$ and hence $E^2$ don't have any additional isogenies. If $n > 4$, we can't have that $E^2(\Q_{2,\infty})_\tors \simeq \Z/2\Z\oplus\Z/2n\Z$ by Theorem \ref{p2result}, while when $n =2,3,4$ both cases are possible depending on the class of the discriminant of $E$ modulo squares. Checking the generic elliptic curves with a rational torsion subgroup isomorphic to $\Z/2n\Z$ for $n =2,3,4$ we see that there are infinitely many curves whose discriminant are congruent to 2 mod squares and infinitely many curves whose discriminants are congruent to $-1$ mod squares. In the first case we have that $E^2(\Q_{2,\infty})_\tors \simeq \Z/2\Z\oplus\Z/2n\Z$, while in the second case $E^2(\Q_{2,\infty})_\tors \simeq \Z/2n\Z$. So all that remains is to check the case when $G = \Z/2\Z \oplus \Z/2\Z$. 

To finish the last case we give a non-isotrivial family $E_t$ with $E_t(\Q)_\tors \simeq \Z/2\Z$ and $\Q(E_t[2]) = \Q(\sqrt{2})$. This family is 
\[
E_t: y^2 = x^3 - \frac{2}{t^2 - 1/2}x^2 - \frac{2}{t^2 - 1/2}x
\]
and generically these curves have no other isogenies and so for infinitely many of them $E_t(\Q_{2,\infty})_\tors \simeq \Z/2\Z\oplus\Z/2\Z$.

\end{proof}

We list in Table \ref{Q2examples} examples of elliptic curves with minimal discriminant achieving growth to each possible torsion group over $\Q_{\infty,2}$.

\begin{remark}
Clearly it is impossible for an elliptic curve to have its torsion ``grow'' and become trivial so in Theorem \ref{thm:p2infinite} $G$ cannot be the trivial group and since an elliptic curve can only go from having trivial 2-torsion to having a point of order 2 in an extension degree divisible by 3, it cannot be that torsion grows to $\Z/2\Z$ over $\Q_{\infty,2}$.
\end{remark}

Now we consider the cases of torsion growth over $\Q_{\infty,3}$. By the results in Section \ref{proofofp3result} we need to consider the cases of torsion growth listed on Table \ref{Q3examples} (we offer an example of minimal conductor for each type of growth in question). We prove that besides $\Z/3\Z$ to $\Z/21\Z$ and $\Z/3\Z$ to $\Z/27\Z$ (which are easily explained by $X_0(21)$ and $X_0(27)$ having finitely many rational points), all of these cases occur for infinitely many $j$-invariants.

First we have a theorem that gives the conductor of a cyclic cubic number field in terms of its defining polynomial.

\begin{tm}[\cite{cubiccond}] \label{cubicconductor}
Let $K$ be a number field with $[K: \Q]=3$ and $\Gal(K/\Q) \simeq \Z/3\Z$. Then $K = \Q(\theta)$ for a $\theta$ satisfying $\theta^3+A\theta+B=0$ with $A,B \in \mathbb{Z}$ and, for any $R\in \mathbb{Z}$, if $R^2 \mid A$ and $R^3 \mid B$, then $|R|=1$. Further, the conductor $\mathfrak{f}(K)$ is given by
\[
\mathfrak{f}(K) = 3^{\alpha} \prod\limits_{\substack{p \text{(prime)} \equiv 1 \mod 3 \\ p \mid (A,B)}} p
\]
where, letting $C$ be the square root of the discriminant of $K$,
\[
\alpha = \begin{cases}
0 & \text{ if } 3 \nmid A \text{ or } 3 \mid\mid A,\;\; 3 \nmid B,\;\; 3^3 \mid C \\
2 & \text{ if } 3^2 \mid\mid A,\;\; 3^2 \mid\mid B \text{ or } 3 \mid\mid A,\;\; 3 \nmid B,\;\; 3^2 \mid\mid C.
\end{cases}
\]
\end{tm}

The following lemma gives a way to construct an elliptic curve with torsion growth $\{\mathcal{O}\}$ over $\Q$ to $\Z/2\Z \oplus \Z/2\Z$ over $\Q_{\infty,3}$ and also an elliptic curve with torsion growth $\Z/3\Z$ over $\Q$ to $\Z/2\Z \oplus \Z/6\Z$ over $\Q_{\infty,3}$.

\begin{lemma}\label{z3toz2xz6criteria}
Let 
\[j_3(h) = \frac{(h+27)(h+3)^3}{h}.\]
Suppose that we have $u,v \in \Z$ with $(u,v)=1$ and $u^2+27v^2 = 4 \cdot 3^k \cdot p^3$ for some $k = 2,3$ and some $p \equiv 1 \pmod 3$. Then there is an elliptic curve $E/\Q$ with $j$-invariant $j_3(\frac{u^2}{v^2})$ such that
\[
E(\Q)_{\text{tors}} = \{ \mathcal{O}\} \;\;\text{ and }\;\; E(\Q_{\infty,3})_{\text{tors}} \simeq \Z/2\Z \oplus \Z/2\Z.
\]
Moreover there is a quadratic twist $E'$ of $E$ such that 
\[E'(\Q)_{\text{tors}} \simeq \Z/3\Z \;\;\text{ and }\;\; E'(\Q_{\infty,3})_{\text{tors}} \simeq \Z/2\Z \oplus \Z/6\Z.\]
\end{lemma}

\begin{proof}
Let $E/\Q$ be an elliptic curve with $E(\Q)_{\text{tors}} \simeq \Z/3\Z$ and $E(\Q_{\infty,3})_\tors \simeq \Z/2\Z \oplus \Z/6\Z$. Then $E$ has a rational 3-isogeny and square discriminant, so $E$ corresponds to a rational point on $X_0(3)$, and so
\[
j(E) = \frac{(h+27)(h+3)^3}{h},
\]
for some $h \in \Q$. A model for an elliptic curve with such a $j$-invariant is given by
\[
E_h : y^2 = f(x) = x^3 + \frac{-27(h+3)^3(h+27)}{(h^2+18h-27)^2} x + \frac{54(h+3)^3(h+27)}{(h^2+18h-27)^2},
\]
and by computing the discriminant of this model we can see that $E_h$ has square discriminant if and only if $h \in (\Q^{*})^2$. If we choose an $h$ such that the discriminant of $E_h$ is a square, then $\Gal(\Q(E_h[2])/\Q) \simeq \Z/3\Z$ and $\Q(E_h[2]) = \Q(f)$.

Now we examine $\Q(f)$ to determine when $\Q(f) = \Q_{1,3}$. This will occur precisely when the conductor of $\Q(f)$ is divisible only by 3. Through a change of variables we see that $\Q(f) = \Q(x^3 + A(h) x + B(h))$ where
\[A(h) = -27(h+3)(h+27) \text{ and } B(h)=54(h+27)(h^2+18h-27).\]
We homogenize the equations by letting $h = \frac{u^2}{v^2}$  be written in lowers terms, so that all parameters are integers to obtain
\[A(u,v) = -27(u^2+3v^2)(u^2+27v^2) \text{ and } B(u,v) = 54(u^2+27v^2)(u^4+18u^2v^2-27v^4).\]
By Theorem \ref{cubicconductor} the conductor will be a power of 3 when the gcd of $A(u,v)$ and $B(u,v)$ is divisible only by 3, primes $p \equiv 2 \mod 3$, and by cubes, since by a change of variables we can remove cubes from the gcd of $A(u,v)$ and $B(u,v)$.

We can see that $(A(u,v),B(u,v)) = 2^a \cdot 3^b \cdot (u^2+27v^2)$ for some $a,b \in \Z^{\geq 0}$, since if a prime divides both $u^2+3v^2$ and $u^4+18u^2v^2-27v^4$ then it must be 2 or 3. Thus, if we choose $u$ and $v$ as in the statement of the lemma, we see that $\Q(E_h[2])$ will have conductor a power of 3, and thus $\Q(E_h[2]) = \Q_{1,3}$.

Finally, by construction, $E_h$ is defined over $\Q$ and has a rational $3$-isogeny. Thus, there is a quadratic twist of $E_h$ that has a 3-torsion point over $\Q$. Note that taking a quadratic twist does not change the field of definition of the 2-torsion points, so this twist indeed has the growth we are looking for over $\Q_{1,3}$.

\end{proof}

Now, we have a lemma to ensure there are infinitely many non-isomorphic $E/\Q$ with the above torsion growth.

\begin{lemma}\label{z3toz2xz6infinite}
For any prime $p \equiv 1 \pmod 3$ and $k=2,3$ there exist $u,v \in \Z$ with $(u,v)=1$ such that $u^2+27v^2 = 4 \cdot 3^k \cdot p^3$.
\end{lemma}

\begin{proof}
Let $K = \Q(\sqrt{-3})$. Then
\[
u^2 + 27v^2 = \operatorname{Nm}_{K/\Q}(u+3v\sqrt{-3})
\]
and so we wish to prove that there are elements of the form $u+3v\sqrt{-3}$ with $(u,v)=1$ of norm $4 \cdot 3^k \cdot p^3$ for any $p \equiv 1 \bmod 3$ and for $k=2, 3$. Since norms are multiplicative, and we see that
\[
4\cdot 3^2 = \operatorname{Nm}_{K/\Q}(3+3\sqrt{-3})
\]
and
\[
4 \cdot 3^3 = \operatorname{Nm}_{K/\Q}(9+3\sqrt{-3}) 
\]
it remains to show that there is an element of norm $p^3$ in $K$.

Let $\alpha$ be a root of $x^2+x+1$, so that the ring of integers of $K$ is equal to $\Z[\alpha]$. Since $p \equiv 1 \bmod 3$, this prime splits in $K$ and so 
\[p = \mathfrak{p} \overline{\mathfrak{p}} = (x+y\alpha)(x+y\alpha^2)\]
for some $x,y \in \Z$. We have $\operatorname{Nm}_{K/\Q}(\mathfrak{p}) = p$ and so we claim $\mathfrak{p}^3$ is the element we want to take. We can find relatively prime $a,b \in \Z$ such that 
\[
\mathfrak{p}^3 = a + 3b\alpha.
\]
Indeed, simply writing $\mathfrak{p}^3 = (x + y \alpha)^3$ for some $x,y \in \Z$ and expanding shows that the coefficient of $\alpha$ is divisible by 3. Further, if $d=(a,b)>1$, then
\[
\mathfrak{p}^3 = (d)\left(\frac{a}{d} + 3\frac{b}{d}\alpha\right).
\]
However, taking norms on both sides shows that $d=p$, so then $\overline{\mathfrak{p}} \mid \mathfrak{p}^3$, which is impossible.

Thus, we have found elements in $K$ of norm $4 \cdot 3^2$, $4 \cdot 3^3$, and $p^3$ for any prime $p \equiv 1 \bmod 3$. The lemma follows from expanding the product of $3 + 3\sqrt{-3}$ and $9+3\sqrt{-3}$ with $\mathfrak{p}^3$ to show that the product is indeed of the form $u + 3v\sqrt{-3}$ with $(u,v)=1$.


\end{proof}

An immediate corollary of this lemma is:

\begin{corollary}
There are infinitely many $j \in \Q$ such that there exists an elliptic curve $E/\Q$ with $j(E)=j$ that satisfies $E(\Q)_{\text{tors}} = \{ \mathcal{O}\}$ and $E(\Q_{\infty,3})_{\text{tors}} \simeq \Z/2\Z \oplus \Z/2\Z$.

There are infinitely many $j \in \Q$ such that there exists an elliptic curve $E/\Q$ with $j(E)=j$ that satisfies $E(\Q)_{\text{tors}} \simeq \Z/3\Z$ and $E(\Q_{\infty,3})_{\text{tors}} \simeq \Z/2\Z \oplus \Z/6\Z$.
\end{corollary}


Now we illustrate a parallel idea for torsion growth from $\{ \mathcal{O} \}$ to $\Z/7\Z$ and from $\Z/3\Z$ to $\Z/9\Z$ over $\Q_{\infty,3}$, respectively.

\begin{lemma}\label{trivtoz7criteria}
Let
\[
j_7(h)=\frac{(h^2+13h+49)(h^2+5h+1)^3}{h}.
\]
Suppose that we have $u,v \in \Z$ with $(u,v)=1$ and $u^2+13uv+49v^2 = 3^k \cdot p^3$ for some $k=2,3$ and some $p \equiv 1 \pmod 3$. Then there is an elliptic curve $E/\Q$ with $j$-invariant $j_7(\frac{u}{v})$ such that $E(\Q)_{\text{tors}} = \{ \mathcal{O} \}$ and $E(\Q_{\infty,3})_\tors \simeq \Z/7\Z$.
\end{lemma}

\begin{proof}

Let $E/\Q$ be an elliptic curve with $E(\Q_{\infty,3})_{\text{tors}} \simeq \Z/7\Z$. Then $E/\Q$ has a 7-isogeny over $\Q$, so it corresponds to a rational point on $X_0(7)$, and so
\[
j(E)=\frac{(h^2+13h+49)(h^2+5h+1)^3}{h}
\]
for some $h \in \Q$. A model for an elliptic curve with such a $j$-invariant is given by
\[
E_h : y^2 = x^3 - \frac{27(h^2+5h+1)^3(h^2+13h+49)}{(h^4+14h^3+63h^2+70h-7)^2}x + \frac{54(h^2+5h+1)^3(h^2+13h+49)}{(h^4+14h^3+63h^2+70h-7)^2}.
\]
We can compute the $7^{\text{th}}$ division polynomial of $E_h$, and obtain that it has one irreducible factor of degree 3, which we denote by $f_7$, and one irreducible factor of degree 21. We wish to determine for which values of $h$ does this degree 3 factor define the extension $\Q_{1,3}$.

By a change of coordinates we see that $\Q(f_3(h))=\Q(x^3+A(h)x+B(h))$ where
\[
A(h) = -3(h^2+13h+49) \text{ and } B(h)=-(2h+13)(h^2+13h+49).
\]
We homogenize the equations by letting $h=\frac{u}{v}$ with $(u,v)=1$ so that all parameters are integers to obtain
\[
A(u,v) = -3(u^2+13uv+49v^2) \text{ and } B(u,v)= -(2u+13v)(u^2+13uv+49v^2).
\]
By Theorem \ref{cubicconductor} the conductor will be a power of 3 when the gcd of $A(u,v)$ and $B(u,v)$ is divisible only by 3, primes $p \equiv 2 \mod 3$, and by cubes, since by a change of variables, we can remove cubes from the gcd of $A(u,v)$ and $B(u,v)$.

We can see that $(A(u,v),B(u,v))=2^a \cdot 3^b \cdot (u^2+13uv+49v^2)$ for some $a,b \in \Z^{\geq 0}$. Thus, if we choose $u$ and $v$ as in the statement of the lemma, we see that $\Q(f_7)$ will have conductor a power of 3, and thus $\Q(f_7) = \Q_{1,3}$. Now taking an appropriate quadratic twist, we can make $\Q(f_7) = \Q(P)$ for a point $P \in E[7]$ of order 7.

\end{proof}

\begin{lemma}\label{z3toz9criteria}
Let $h = \frac{u}{v}$ for $u,v \in \Z$ with $(u,v)=1$ satisfying
\[
u^2 + 3uv + 9v^2 = 3^3 \cdot p^3
\]
for some prime $p \equiv 1 \pmod 3$ and let $E_h/\Q$ be the elliptic curve given by
\[
E_h : y^2 = x^3 - 27h^5(h^3-24)^5 x + 54h^6(h^3-24)^6(h^6 - 36h^3+216).
\]
Then $E_h(\Q)_{\text{tors}} \simeq \Z/3\Z$ and $E_h(\Q_{\infty,3})_{\text{tors}} \simeq \Z/9\Z$.
\end{lemma}

\begin{proof}

For such torsion growth to occur, an elliptic curve must have a 3-torsion point over $\Q$ as well as a 9-isogeny over $\Q$ whose kernel contains this point of order 3. A model for elliptic curves over $\Q$ with this level structure is given in \cite{DDH} Table 6, and they are elliptic curves precisely of the form
\[
E_h : y^2 = x^3 - 27h^5(h^3-24)^5 x + 54h^6(h^3-24)^6(h^6 - 36h^3+216)
\]
for some $h \in \Q$. We can compute the $9^{th}$ division polynomial of $E_h$ and divide it by the $3^{rd}$ division polynomial and obtain one factor of degree 3, which we denote by $f_9$. We wish to determine for which values of $h$ does this degree 3 factor define the extension $\Q_{1,3}$.

By a change of coordinates we see that $\Q(f_9(h)) = \Q(x^3 + A(h)x + B(h))$ where
\[
A(h) = - 432(h^2+3h+9) \;\;\text{ and }\;\; B(h) = -1728(2h+3)(h^2+3h+9).
\]
We homogenize the equations by letting $h = \frac{u}{v}$ with $(u,v)=1$ so that all parameters are integers to obtain
\[
A(u,v) = - 432(u^2+3uv+9v^2) \;\;\text{ and }\;\; B(u,v) = -1728(2u+3v)(u^2+3uv+9v^2).
\]
By Theorem \ref{cubicconductor} the conductor will be a power of 3 when the gcd of $A(u,v)$ and $B(u,v)$ is divisible only by 3, primes $p \equiv 2 \mod 3$,  and by cubes, since by a change of variables, we can remove cubes from the gcd of $A(u,v)$ and $B(u,v)$.

We can see that $(A(u,v),B(u,v)) = 2^a \cdot 3^b \cdot (u^2 + 3uv + 9v^2)$ for some $a,b \in \Z^{\geq 0}$. Thus, if we choose $u$ and $v$ as in the statement of the lemma, we see that $\Q(f_9)$ will have conductor a power of 3, and thus $\Q(f_9) = \Q_{1,3}$.  Now, if $P$ is a generator of the kernel of the isogeny, then $G_{\Q}$ acts on $\langle P \rangle$ by multiplication by 1, 4, or 7, since $3P \in E(\Q)$.  Hence, $P$ is fixed by an index 3 subgroup of $G_{\Q}$, and thus defined over a cubic field, in particular, the cubic field where its $x$-coordinate is defined, i.e. $\Q(f_9) = \Q_{1,3}$.
\end{proof}

We remark that the criteria given in Lemma \ref{trivtoz7criteria} and Lemma \ref{z3toz9criteria} are asking when $3^3 p^3$ is primitively represented by some binary quadratic form. In both Lemmas, the binary quadratic forms have discriminant $-27$. We now prove that both the above criteria are satisfied for infinitely many primes $p$ by proving a statement about integers represented by binary quadratic forms of discriminant $-27$.

\begin{lemma}
Let $f(x,y)$ be any binary quadratic form of discriminant $-27$. Then there exist primitive solutions to $f(x,y) = 3^3 \cdot p^3$ for all $p \equiv 1 \pmod 3$.
\end{lemma}

\begin{proof}

Since $f(x,y)$ has discriminant $-27$, there is an $\SL_2(\Z)$ transformation of variables so that 
\[
f(x,y) \sim u^2 + uv + 7v^2,
\]
i.e. $\begin{bmatrix}
u \\ v 
\end{bmatrix} = M \begin{bmatrix}
x \\ y
\end{bmatrix}$ for some $M \in \SL_2(\Z)$. Moreover, $\SL_2(\Z)$ transformations preserve the gcd of the coordinates, so we need only find primitive solutions to
\[
u^2 + uv + 7v^2 = 3^3 \cdot p^3
\]
for all primes $p \equiv 1 \pmod 3$. We let $K = \Q(\sqrt{-3})$ with ring of integers by $\Z[\alpha]$ where $\alpha = \frac{-1+\sqrt{-3}}{2}$ is a primitive $3^{rd}$ root of unity. Notice that
\[
u^2 + uv + 7v^2 = \operatorname{Nm}_{K/\Q}(u + 3v\alpha),
\]
so we want to find elements in $K$ of the form $u+3v\alpha$ with $(u,v)=1$ and norm $3^3 \cdot p^3$.

Let $p$ be a prime such that $p \equiv 1 \pmod 3$. Then $p$ is split in $K$, so
\[
p = \mathfrak{p} \overline{\mathfrak{p}} = (x + y \alpha)(x+y\alpha^2)
\]
for some $x,y \in \Z$. Recalling the argument given in the proof of Lemma \ref{z3toz2xz6infinite},  
\[
\mathfrak{p}^3 = a + 3b\alpha
\]
for some relatively prime integers $a,b \in \Z$. Further, we have that $3$ is ramified in $K$,
\[
(3) = \mathfrak{p}_3^2 = (1 + 2\alpha)^2
\]
and finally that $\mathfrak{p}_3^3 = (3 + 2(3\alpha))$. Thus, an element of norm $3^3 p^3$ is
\[
\mathfrak{p}_3^3\mathfrak{p}^3
\]
and the lemma follows from expanding the product of the two elements to show that the product is indeed of the form $u + 3v\alpha$ with $(u,v)=1$.
\end{proof}

Thus, an immediate corollary is

\begin{corollary}
There are infinitely many $j \in \Q$ such that there exists an elliptic curve $E/\Q$ with $j(E)=j$ that satisfies $E(\Q)_{\text{tors}} = \{ \mathcal{O}\}$ and $E(\Q_{\infty,3})_{\text{tors}} \simeq \Z/7\Z$.

There are infinitely many $j \in \Q$ such that there exists an elliptic curve $E/\Q$ with $j(E)=j$ that satisfies $E(\Q)_{\text{tors}} \simeq \Z/3\Z$ and $E(\Q_{\infty,3})_{\text{tors}} \simeq \Z/9\Z$.
\end{corollary}


\begin{table}
\begin{center}

\begin{tabular}{|r|c|c|}\hline
Cremona Reference & $E(\Q)_\tors$ & $E(\Q_{\infty,2})_\tors$ \\\hline
\href{http://www.lmfdb.org/EllipticCurve/Q/704d1}{\texttt{704d1}}&$\{\mathcal{O}\}$&$\ZZ/3\ZZ$ \\\hline
\href{http://www.lmfdb.org/EllipticCurve/Q/24a6}{\texttt{24a6}}&$\ZZ/2\ZZ$ &$\ZZ/4\ZZ$ \\\hline
\href{http://www.lmfdb.org/EllipticCurve/Q/704a1}{\texttt{704a1}}&$\{\mathcal{O}\}$&$\ZZ/5\ZZ$ \\\hline
\href{http://www.lmfdb.org/EllipticCurve/Q/320c1}{\texttt{320c1}}&$\ZZ/2\ZZ$ &$\ZZ/6\ZZ$ \\\hline
\href{http://www.lmfdb.org/EllipticCurve/Q/832f1}{\texttt{832f}}&$\{\mathcal{O}\}$&$\ZZ/7\ZZ$ \\\hline
\href{http://www.lmfdb.org/EllipticCurve/Q/24a3}{\texttt{24a3}}&$\ZZ/4\ZZ$ &$\ZZ/8\ZZ$ \\\hline
\href{http://www.lmfdb.org/EllipticCurve/Q/1728j3}{\texttt{1728j3}}&$\{\mathcal{O}\}$&$\ZZ/9\ZZ$ \\\hline
\href{http://www.lmfdb.org/EllipticCurve/Q/768b1}{\texttt{768b1}}&$\ZZ/2\ZZ$ &$\ZZ/10\ZZ$ \\\hline
\href{http://www.lmfdb.org/EllipticCurve/Q/30a5}{\texttt{30a5}}&$\ZZ/6\ZZ$ &$\ZZ/12\ZZ$ \\\hline
\href{http://www.lmfdb.org/EllipticCurve/Q/14a5}{\texttt{14a5}}&$\ZZ/2\ZZ$ &$\ZZ/2\ZZ\oplus\ZZ/2\ZZ$ \\\hline
\href{http://www.lmfdb.org/EllipticCurve/Q/24a2}{\texttt{24a2}}&$\ZZ/2\ZZ\oplus\ZZ/2\ZZ$ &$\ZZ/2\ZZ\oplus\ZZ/4\ZZ$ \\\hline
\href{http://www.lmfdb.org/EllipticCurve/Q/14a2}{\texttt{14a2}}&$\ZZ/6\ZZ$ &$\ZZ/2\ZZ\oplus\ZZ/6\ZZ$ \\\hline
\href{http://www.lmfdb.org/EllipticCurve/Q/32a4}{\texttt{32a4}}&$\ZZ/4\ZZ$ &$\ZZ/2\ZZ\oplus\ZZ/8\ZZ$ \\\hline
\end{tabular}
\caption{Elliptic curves of minimal conductor with torsion growth over $\Q_{\infty,2}$.}
\label{Q2examples}
\end{center}
\end{table}

\begin{table}
\begin{center}
\begin{tabular}{|r|c|c|}\hline
Cremona Reference & $E(\Q)_\tors$ & $E(\Q_{\infty,3})_\tors$ \\\hline
\href{http://www.lmfdb.org/EllipticCurve/Q/162b2}{\texttt{162b2}} & $\{ \mathcal{O}\}$ & $\ZZ/7\ZZ$ \\\hline
\href{http://www.lmfdb.org/EllipticCurve/Q/324a2}{\texttt{324a2}} & $\{ \mathcal{O}\}$ & $\ZZ/2\ZZ \oplus \ZZ/2\ZZ$ \\\hline
\href{http://www.lmfdb.org/EllipticCurve/Q/27a3}{\texttt{27a3}} & $\ZZ/3\ZZ$ & $\ZZ/9\ZZ$ \\\hline
\href{http://www.lmfdb.org/EllipticCurve/Q/162b1}{\texttt{162b1}} & $\ZZ/3\ZZ$ & $\ZZ/21\ZZ$ \\\hline
\href{http://www.lmfdb.org/EllipticCurve/Q/27a4}{\texttt{27a4}} & $\ZZ/3\ZZ$ & $\ZZ/27\ZZ$ \\\hline
\href{http://www.lmfdb.org/EllipticCurve/Q/324a1}{\texttt{324a1}} & $\ZZ/3\ZZ$ & $\ZZ/2\ZZ \oplus \ZZ/6\ZZ$ \\\hline
\end{tabular}
\caption{Elliptic curves of minimal conductor with torsion growth over $\Q_{\infty,3}$.}
\label{Q3examples}
\end{center}
\end{table}

\end{document}